\documentclass[11pt]{amsart}
\usepackage{amssymb,amsmath}
\usepackage[utf8]{inputenc} 
\usepackage[T1]{fontenc}

\title[Distinguishing number of Urysohn metric spaces]{Distinguishing number of universal homogeneous Urysohn metric spaces}

\author[A. Bonato]{Anthony Bonato}
\thanks{The authors gratefully acknowledge support from NSERC}
\address{A. Bonato: Department of Mathematics, Ryerson University, 350 Victoria St., Toronto, ON, Canada, M5B 2K3}
\email {abonato@ryerson.ca}

\author[C. Laflamme] {Claude Laflamme}
\address{C. Laflamme: University of Calgary, Department of Mathematics and Statistics, Calgary, Alberta, Canada T2N 1N4}
\email {laflamme@ucalgary.ca}

\author[M. Pawliuk]{Micheal Pawliuk}
\address{M. Pawliuk: University of Toronto Mississauga, Mathematical and Computational Sciences, Mississauga, Ontario, Canada L5L 1C6}
\email {m.pawliuk@mail.utoronto.ca}

\author [N. Sauer]{Norbert Sauer}
\address{N. Sauer: University of Calgary, Department of Mathematics and Statistics, Calgary, Alberta, Canada T2N 1N4}
\email{nsauer@ucalgary.ca}

\newtheorem{defin}{Definition}[section]
\newtheorem{thm}[defin]{Theorem}
\newtheorem*{thm*}{Theorem}
\newtheorem{lem}[defin]{Lemma}
\newtheorem{coroll}[defin]{Corollary}

\newtheorem{fact}[defin]{Fact}

\newtheorem*{obs*}{Observation}

\newtheorem*{claim*}{Claim}

\newtheorem*{problem*}{Open Problem}

\newtheorem*{conj*}{Conjecture}

\newcommand{\Fra}{Fra\"{\i}ss\'e }

\newcommand{\Reals}{\mathbb R}
\newcommand{\SG}{\boldsymbol{G}} %structure
\newcommand{\ST}{\boldsymbol{S}} %structure
\newcommand{\U}{\boldsymbol{U}}
\newcommand{\US}{\boldsymbol{U}_{\negthinspace\mathcal{S}}}
\newcommand{\RS}{\boldsymbol{R}} % structure R
\newcommand{\MS}{\boldsymbol{M}} % metric space M
\newcommand{\NS}{\boldsymbol{N}} % metric space N
\newcommand{\AS}{\boldsymbol{A}} % metric space A
\newcommand{\BS}{\boldsymbol{B}} % metric space B
\newcommand{\CS}{\boldsymbol{C}} % metric space C
\newcommand{\DS}{\boldsymbol{D}} % metric space D
\newcommand{\TS}{\boldsymbol{T}} % metric space T

\newcommand{\dis}{D}
\newcommand{\SetS}{\mathcal{S}} % set S of reals
\newcommand{\spec}{\mathrm{spec}}
\newcommand{\dist}{d} %metric
\newcommand{\Aut}{\mathrm{Aut}}
\newcommand{\noi}{\noindent}

\date{\today}

\begin{document}


\begin{abstract}  
The distinguishing number of a structure is the smallest size of a partition of its elements so that only the trivial automorphism of the structure preserves each cell of the partition. We show that for any countable subset of the positive real numbers, the corresponding countable homogeneous Urysohn metric space, when it exists, has distinguishing number 2 or the distinguishing number is infinite. 

While it is known that a sufficiently large finite primitive structure has distinguishing number 2, unless its automorphism group is  the full symmetric group or alternating group, the infinite case is open and these countable Urysohn metric spaces provide further confirmation toward the conjecture that all primitive homogeneous countably infinite structures have distinguishing number 2 or else  the distinguishing number is infinite. 
\end{abstract}

\maketitle
%%%%%%%%%%%%%%%

\section{Introduction}

The \emph{asymmetric colouring number} of a graph was introduced by Babai long ago in \cite{B76}, and it resurfaced more recently as the \emph{distinguishing number} in  the work of Albertson and Collins in \cite{AC96}. First call a group of permutations $\SG$ on a set $A$ \emph{$k$-distinguishable} if there exists a partition of $A$ into $k$ cells such that only the identity permutation in $\SG$ fixes setwise all of the cells of the partition. It is evident that $\SG$ is always $|A|$-distinguishable. The least cardinal number $k$ such that $\SG$ is $k$-distinguishable is its \emph{distinguishing number}  $\dis(\SG)$. We call a graph or any structure $\ST$  $k$-distinguishable if its automorphism group $\Aut(\ST)$ satisfies  $\dis(\Aut(\ST)) \leq k$.

%The \emph{distinguishing number} can generally be defined for any structure $\ST$ as the smallest positive integer $d$ such that a partition of its elements in $d$ pieces exists so that only the trivial automorphism of $\ST$ preserves the partition, and is denoted by either $\dis(\ST)$. 

The distinguishing number is the amount of symmetry found within a structure, leading to interesting structural information that comes from the investigation of what is needed to break that symmetry. Of particular interest to us are countable homogeneous structures, carrying any two finite substructures of the same size and type into each other.  Their  automorphism group is thus highly symmetric and in particular transitive as a permutation group.
The set of rational numbers with its linear order relation is such a countable homogeneous structure, and is easily seen to have infinite distinguishing number. On the other hand, the Rado graph (or infinite random graph) is also homogeneous, but  Imrich, Klavzar and Trofimov showed in \cite{IKT07} that its distinguishing number is 2, which is the smallest it can be because the Rado graph is not rigid. The distinguishing number of various other finite and countable homogeneous structures was determined in \cite{bd1,bd2,LNS10}, including all simple and directed homogeneous graphs and posets.  In particular in all cases of infinite homogeneous simple and directed graphs, it was shown in \cite{LNS10} that their distinguishing number is either 2 or infinite, with only obvious exceptions having imprimitive automorphism groups. The following was thus conjectured.

\begin{conj*}[\cite{LNS10}]
 The distinguishing number of all primitive homogeneous countably infinite structures is 2 or infinite.
\end{conj*}

The conjecture is very much in the spirit of the finite case, where Cameron, Neumann and Saxl proved (see \cite{BC11})  that a sufficiently large, finite primitive permutation group has distinguishing number 2, unless it is the full symmetric group or alternating group. The 43 exceptions were determined by Seress, and one of these exceptions is the dihedral group $D_{10}$, which is primitive and the automorphism group of the homogeneous graph $C_5$; it has distinguishing number 3 and hence  the necessity for the conjecture to address only infinite structures. A tool developed in \cite{LNS10} appears in the right direction to confirm the conjecture, namely, that of a fixing type for the action of a group $G$ on a set $A$. If the action does have such a fixing type, then the distinguishing number of $G$ acting on $A$ is 2. It may be possible that a more general result for all primitive groups exists (homogeneous or not), but we have no insight in that direction. 

A graph having distinguishing number 2 has an interesting translation to permutation group theoretic properties of its automorphism group, see Section 2.2 of \cite{BC11}. Hence it is reasonable to hope that the 2-distinguishability of the rational Urysohn space, and the distinguishing numbers of other Urysohn spaces addressed below, have interesting translations into properties of their automorphism groups.

%Of particular interest is that such extremely amenable groups are exactly those automorphism groups of Fra\"iss\'e limits (hence, homogeneous structures) of a Fra\"iss\'e ordered class with the Ramsey property. See \cite{f} for background on Fra\"iss\'e limits.

In this paper, we consider the case of homogeneous countable \emph{Urysohn metric spaces} $\US$ for a given countable spectrum $S \subseteq \Reals +$, constructed as the Fra\"iss\'e limits of all finite metric spaces whose spectrum is a subset of $S$. 
Note that not every subset $S$ can be the spectrum of such a Urysohn metric space, and a necessary and sufficient condition is known as the ``4-values'' condition, which is precisely when metric triangles amalgamate; when this is the case we call $S$ a \emph{universal spectrum}. Depending on $S$, we will see that the automorphism group of $\US$ may or may not be primitive; we do not have a characterization for a spectrum to yield a primitive automorphism group of its corresponding Urysohn space.

These countable metric spaces are very much related to the well known (uncountable) Urysohn space,  the complete separable metric space which is both homogeneous and universal; it is the completion of the countable homogeneous Urysohn spaces using the rationals as spectrum. See, for example,  \cite{cam,mell}. 

The main result of the paper is as follows.

\begin{thm*}[Main Theorem]\label{thm:mainClassification}
Let $\SetS \subseteq \Reals +$ be a countable universal spectrum and $\US$ the countable homogeneous structure with spectrum $S$. We then have that  $\dis(\US)=2 \mbox{ or } \omega$, and the following items hold.
\begin{enumerate}
\item If $\SetS$ has a positive limit (not necessarily in $\SetS$), then $\dis(\US)=2$.
\item If $\SetS$ has no positive limits but has 0 as a limit, then $\dis(\US) = 2$ if and only if $S$ contains arbitrarily large elements of arbitrarily small distance.
\item If $\SetS$ does not have a limit,  then $\dis(\US)=2$ if and only if $S$ contains two elements of distance smaller than or equal to the minimum positive element of $\SetS$.
\end{enumerate}
\end{thm*}

The proof is the result of analyzing the existence of various limit structures for the set $\SetS$ along the following lines, and the necessary lemmas will be proved in the remaining sections of the paper. 

\begin{proof}
Let $\SetS$ be a universal spectrum. If $\SetS$ has a positive limit (not necessarily in $\SetS$), then $\dis(\US)=2$ by Lemma~\ref{lem:poslimno2}.
 If $\SetS$ has no positive limit but has 0 as a limit, then $\dis(\US)=2$ or  $\omega$ by Lemma \ref{lem:infd0}.
Moreover $\dis(\US)=2$ if and only if $\mathrm{gap}(\SetS_{\geq s})=0$ for some positive number $s\in \SetS$; as $\SetS$ is assumed to have no positive limit, this latter condition is equivalent to  $\SetS$ containing arbitrarily large elements of arbitrarily small distance.

Finally, if $\SetS$ does not have a limit, then by Lemma \ref{lem:nolim} $\dis(\US)=2$ if and only if there exist numbers $a< b \in \SetS$ with $b-a \leq \min(\SetS_{>0})$, that is  contains two elements of distance smaller than the minimum positive element of $\SetS$.
\end{proof}

It follows that the well known \emph{rational Urysohn space}, that is $\US$ for $\SetS$ consisting of the non-negative rational numbers, has distinguishing number 2. Moreover we shall see following Theorem \ref{thm:4valassoc} that various examples exist showing that all those cases of the Main Theorem do occur.  \\

It is worth noting that the rigid subspaces which are used in the proofs of those necessary lemmas are almost always rigid  forests determined by the $s$-distance graph of a metric space for some $s$ in the spectrum: there is an edge between $x,y \in \US$ if and only if the distance between $x$ and $y$ is $s$. \\

We conclude the introduction by drawing a parallel with other work.  Imrich et al.\ in \cite{ISTW13} have shown
that a countable permutation group acting on a countable set has distinguishing number 2, under the additional assumption that the group has \emph{infinite motion}, meaning that every non-identity group element moves infinitely many elements. The automorphism groups of Urysohn spaces in this paper have infinite motion, but are uncountable. 
Moreover recall that, up to (topological group) isomorphism, the closed subgroups of the infinite symmetric $S_{\infty}$ are exactly the automorphism groups of countable structures, and thus our work can be viewed in that setting. The work of Imrich et al. in \cite{ISTW13} is also focused on closed subgroups of $S_{\infty}$, where in particular they conjecture that any closed subgroup of $S_{\infty}$ having infinite motion and where all orbits of its point stabilizers are finite has distinguishing number 2; this is the so-called \emph{Infinite Motion Conjecture for Permutation Groups}. However the  point stabilizers of automorphism groups of homogeneous Urysohn spaces all are infinite. \\

\section{General Notions and Preliminaries}

A relational structure $\RS$ is {\em rigid} if its group of automorphisms $\Aut(\RS)$ consists only of the identity. The idea behind the distinguishing number is to find the smallest number of predicates $\langle P_i:i<d \rangle$ such that the expanded structure $(\RS; P_i:i<d)$ becomes rigid.

For a metric space $\MS = (M,d_{\MS})$, let $\spec ( \MS )$, the {\em spectrum} of $\MS$, be the set of the distances between points of $\MS$. A metric space $\MS$ is {\em universal} if it embeds every finite metric space $\NS$ with spectrum $\spec(\NS)\subseteq \spec(\MS)$.

For $\SetS$ a set of reals and $r \in \Reals$, let $\SetS_{>r}=\{s\in \SetS : s > r\}$, and similarly for $\SetS_{\geq r}$. If $\SetS \subseteq \Reals_{\geq 0}$ is countable, then let $\mathfrak{A}(\SetS)$ denote the set of finite metric spaces whose spectrum is a subset of $\SetS$. Note that $\mathfrak{A}(\SetS)$ is an age (meaning that it is closed under isomorphism and substructures, and up to isomorphism has only countably many members),  and we will need conditions on $\SetS$ for which the age $\mathfrak{A}(\SetS)$ has the amalgamation property.

\begin{defin}\label{defin:amalg}
We call a pair of metric spaces $\AS$ and $\BS$ an {\em amalgamation instance} if $\dist_{\AS}(x,y)=\dist_{\BS}(x,y)$ for all $x,y\in A\cap B$. If so, then we define:
\[
\amalg(\AS,\BS)=\{\CS=(A\cup B; \dist_{\CS}) : \; \CS\restriction {A}=\AS \text{ and } \CS\restriction {B}=\BS\}.
\]
For $\SetS\subseteq \Reals_{\geq 0}$ and $\{\AS,\BS\}\subseteq \mathfrak{A}(\SetS)$ let:
\[
\amalg_\SetS(\AS,\BS)=\{\CS\in \amalg(\AS,\BS): \spec(\CS)\subseteq \SetS\}.
\]
Finally, we say that the set $\SetS \subseteq \Reals_{\geq 0}$ has the {\em amalgamation property} if $\amalg_\SetS(\AS,\BS)\not=\emptyset$ for all amalgamation instances $\{\AS, \BS\} \subseteq \mathfrak{A}(\SetS)$.
\end{defin}

We now define the ``4-values'' condition, which is the description that triangles amalgamate.

\begin{defin}\label{defin:4val}
A set $\SetS\subseteq \Reals_{\geq 0}$ satisfies the {\em 4-values} condition if $\amalg_\SetS(\AS,\BS)\not=\emptyset$ for any two amalgamation instances of the form $\AS=(\{x,y,z\};\dist_{\AS})$ and $\BS=(\{x,y,w\};\dist_{\BS})$ in $ \mathfrak{A}(\SetS)$.
\end{defin}

There are several equivalent definitions of the 4-values condition. The first version has been given in \cite{DLPS} together with Theorem \ref{thm:Frabas} below. In \cite{Sadist} there is an equivalent version of the 4-values condition given and Definition \ref{defin:4val} is stated as Lemma 3.3. Definition \ref{defin:4val} here is the one most suitable for our purposes.

It is evident that  the spectrum of any  metric space satisfies the 4-values condition, and it provides the leading condition for a set $\SetS$ to yield a non-trivial Urysohn metric space. 

\begin{defin}\label{defin:universp}
A countable set $\SetS\subseteq \Reals_{\geq 0}$ is a {\em universal spectrum} if the following items hold.
\begin{enumerate}
\item The element $0$ is in $\SetS$ and $\SetS$ contains at least one positive number; 
\item The set $\SetS$ satisfies the 4-values condition.
\end{enumerate}
\end{defin}

 It follows from Theorem 3.8 of \cite{Sadist} that if $\SetS$ is a universal spectrum, then $\mathfrak{A}(\SetS)$ is an age with amalgamation. Hence, the next  theorem follows from the general \Fra theory.

\begin{thm}\label{thm:Frabas}\cite{f}
If $\SetS\subseteq \Reals_{\geq 0} $ is a universal spectrum, then $\mathfrak{A}(\SetS)$ is an age with amalgamation and there exists a countable homogeneous, universal, metric space $\US$ whose spectrum is $\SetS$.
\end{thm}

Even though by definition an amalgamation instance can be amalgamated, we will in many cases want to do so controlling the new distances. We therefore have the following.

\begin{lem}\label{lem:amalbetter}
Let $\SetS\subseteq \Reals_{\geq 0}$ be a universal spectrum, and $\AS$ and $\BS$ an amalgamation instance in $\mathfrak{A}(\SetS)$. If there exists a number $s\in \SetS$ so that
\[ s\leq \dist_{\AS}(a,x)+\dist_{\BS}(x,b)\]
for all $a\in A\setminus B$, $x\in A\cap B$, and $b\in B\setminus A$, then there exists a metric space $\CS\in \amalg_{\SetS}(\AS,\BS)$ so that $\dist_{\CS}(a,b)\geq s$ for all $a\in A\setminus B$ and $b\in B\setminus A$.
\end{lem}

\begin{proof}
Because $\SetS\subseteq \Reals_{\geq 0}$ satisfies the 4-values condition, there is a metric space $\DS \in \amalg_{\SetS}(\AS,\BS)$. Let $\CS$ be the binary relational structure obtained from $\DS$ by replacing the new distances as follows:
\[ \dist_{\CS}(a,b) = \max\{\dist_{\DS}(a,b),s\} \]
for every $a\in A\setminus B$ and every $b\in B\setminus A$.

We claim that $\CS$ is indeed a metric space in $\mathfrak{A}(\SetS)$, and hence, all triangles of $\CS$ not in $\AS$ and not in $\BS$ must be verified to be metric.

One type of such triangles is of the form $\{a,x,b\}$ with $a\in A\setminus B$ and $b\in B\setminus A$ and $x\in A\cap B$. As a triangle of $\DS$ it is metric, and together with the assumption on $s$ we derive that
\[ \dist_{\CS}(a,b) = \max\{\dist_{\DS}(a,b),s\} \leq \dist_{\AS}(a,x)+\dist_{\BS}(x,b).\]

The other type of triangles we need to verify is of the form either $\{a,a',b\}$ with $\{a,a'\}\subseteq A\setminus B$ and $b\in B\setminus A$, or the other way around of the form $\{a,b,b'\}$ with $a\in A\setminus B$ and $\{b,b'\}\subseteq B\setminus A$; it suffices to consider the first case.

If both $\dist_{\DS}(a,b)\leq s$ and $\dist_{\DS}(a,'b)\leq s$, then $\dist_{\CS}(a,b)= \dist_{\CS}(a',b)=s$, and  $\dist_{\CS}(a,a')= \dist_{\DS}(a,a')\leq \dist_{\DS}(a,b) + \dist_{\DS}(a',b) \leq \dist_{\CS}(a,b)+ \dist_{\DS}(a',b)$ shows that $\{a,a',b\}$ is metric.
If both $\dist_{\DS}(a,b)\geq s$ and also $\dist_{\DS}(a',b)\geq s$, then the side lengths have not changed from $\DS$ to $\CS$ and hence, again metric.

For the remaining case, say $\dist_{\DS}(a,b)<s$ but $\dist_{\DS}(a',b)\geq s$. Then $\dist_{\CS}(a,b)=s \leq \dist_{\DS}(a',b) =\dist_{\CS}(a',b) \leq
\dist_{\CS}(a,a') + \dist_{\CS}(a',b)$. The verification of the other two sides is immediate.
\end{proof}

The following general notions will useful to analyze various universal spectra.

\begin{defin}\label{defin:initjump}
Let $\SetS$ be a universal spectrum. An element $s\in \SetS$ is called:
\begin{enumerate}
\item an {\em initial number} of $\SetS$ if $[{1\over 2}s,s)\cap \SetS=\emptyset$,
\item a {\em jump number} of $\SetS$ if $(s,2s]\cap \SetS=\emptyset$,
\item an {\em insular number} of $\SetS$ if it both an initial and a jump number of $\SetS$.
\end{enumerate}

A subset $B\subseteq \SetS$ is a {\em block} of $\SetS$ if:
\begin{enumerate}
\item $B$ is an interval of $\SetS$,
\item $\min B$ is a positive initial number of $\SetS$,
\item either $B$ is unbounded, or $\max B$ is a jump number of $\SetS$, and
\item $\max B$ (if it exists) is the only jump number (of either $B$ of $S$).
\end{enumerate}
\end{defin}

Hence, if $ s \in \SetS$ is insular, then $B=\{s\}$ is a block consisting of only one element. Note also for future reference that if $s>0$ is a jump number of $\SetS$, then the relation $\stackrel{s}{\sim}$ given by $x\stackrel{s}{\sim}y$ if $\dist(x,y)\leq s$ is an equivalence relation on $\US$, this will soon play an important role. 

\section{Universal spectra without positive limits} 

In this section, we develop tools to handle the case where the spectrum does not have a positive limit. In particular, no positive element $r$ of $\Reals$ is a limit of $\SetS$, whether $r \in \SetS$ or not.

To handle this case, we first call a  set $\SetS \subseteq \Reals_{\geq 0} $ {\em inversely well ordered} if every non-empty bounded above subset of $\SetS$ has a maximum. We recall the $\oplus$ operation from \cite{SA2}, defined there for a closed set of reals, but which for our purpose can also be defined for inversely well ordered sets.

\begin{defin}\label{defin:oplus} 
Let $\SetS$ be an inversely well ordered, and for $r,t \in \SetS$ define
\[
r\oplus t =\max\{s \in \SetS : s \leq r+t \}.
\]
\end{defin}

\noi Observe that if $\{r,s,t\}\subseteq \SetS$, then $r\oplus t\geq s$ if and only if $r+t\geq s$; and hence, we note the following obvious observation which will be used without warning.

\begin{lem}
A triangle is metric if and only if the $\oplus$ sum of any two of the three side lengths is larger than or equal to the third side length.
\end{lem}

\noi The $\oplus$ operation on $\SetS$ can easily be verified to be commutative and monotone. Further it was shown \cite{SA2} that the associativity of the $\oplus$ operation is equivalent to the 4-value condition for a closed set. That assumption, however, was only used to justify the operation and the same argument can be used for the following.

\begin{thm} \label{thm:4valassoc}
If $\SetS$ is an inversely well ordered set, then it satisfies the 4-value condition if and only if the $\oplus$ operation on $\SetS$ is associative.
\end{thm}

This result can be used to easily verify that the following inversely well ordered sets all satisfy the 4-values condition:
\[\begin{array}{lll}
\SetS_1 & = & 0 \cup \{ 1+1/n: n \in \omega \}, \\
\SetS_{2a}  & = & 0 \cup \{ 1 / 2^{2n}: n \in \omega \}, \\
\SetS_{2b} & = &  0 \cup \{ 1 /2^{2n}: n \in \omega \} \cup \{ 2^n,2^n+1/n: n \in \omega\}, \\
\SetS_{3a} & = & \{ 0,1 \}, \\
\SetS_{3b} & = & \{ 0,1,2 \}. 
\end{array}\]

\noi The set $\SetS_1$ has 1 as a limit. Both sets $\SetS_{2a}$ and  $\SetS_{2b}$ have no positive limits but do have 0 as a limit;  $\SetS_{2b}$ has arbitrarily large elements if arbitrarily small distance, while  $\SetS_{2a}$ does not. 
Both sets $\SetS_{3a}$ and  $\SetS_{3b}$ do not have any limit; $\SetS_{3b}$ has two elements of distance smaller than or equal to the minimum positive element, while $\SetS_{3a}$ does not have such elements. Hence this shows that all types universal spectrum of the Main Theorem do occur.  

We note that the homogenous Urysohn space $\U_{\SetS_{3b}}$ is nothing else but the Rado graph, hence providing another (albeit lengthy) proof that its distinguishing number is 2. \\

Our arguments below rely on an analysis of inversely well ordered set, and moreover we use the $\oplus$ operation to construct a specific and controlled amalgamation.

\begin{lem}\label{lem:oplusamalg}
Let $\SetS$ be an inversely well ordered universal spectrum, and $\AS$ and $\BS$ in $\mathfrak{A}(\SetS)$ an amalgamation instance.
Then there exists a unique metric space $\CS\in \amalg_\SetS(\AS,\BS)$, which we denote by $\amalg^\oplus_\SetS(\AS,\BS)$,
 such that:
\[ \dist_{\CS}(a,b)= \min\{\dist_{\AS}(a,x)\oplus \dist_{\BS}(x,b) : \; x\in A\cap B\} \]
for all $a\in A\setminus B$ and $b\in B\setminus A$. 
\end{lem}

\begin{proof}
Let $\CS$ be defined as above, we must show that every triangle of $\CS$ is metric. 

Let $\DS\in \amalg(\AS,\BS)$. Note first that $\dist_{\DS}(a,b)\leq \dist_{\CS}(a,b)$ for all $a\in A\setminus B$ and $b\in B\setminus A$; this is because $ \dist_{\DS}(a,b) \leq \dist_{\AS}(a,x) + \dist_{\BS}(x,b)$ for any $x\in A\cap B$, and therefore, $ \dist_{\DS}(a,b) \leq \dist_{\AS}(a,x) \oplus \dist_{\BS}(x,b)$ by definition of $\oplus$.

The first kind of triangles to consider are of the form $\{a,x,b\}$ for $a\in A\setminus B$, $x\in A\cap B$ and $b\in B\setminus A$. Since $\DS$ is metric and $\dist_{\DS}(a,b)\leq \dist_{\CS}(a,b)$, all we need to verify is the inequality $\dist_{\CS}(a,b) \leq \dist_{\AS}(a,x) + \dist_{\BS}(x,b)$. Let $x'=\mu(a,b)$. By definition, we have that
$\dist_{\CS}(a,b) = \dist_{\AS}(a,x')\oplus\dist_{\BS}(b,x') \leq \dist_{\AS}(a,x')+\dist_{\BS}(b,x') \leq
\dist_{\AS}(a,x)+\dist_{\BS}(b,x)$.

Now, by symmetry, the only other case is a triangle $\{a,a',b\}$ for $\{a,a'\}\subseteq A\setminus B$ and $b\in B\setminus A$. Because $\dist_{\DS}(a,b)\leq \dist_{\CS}(a,b)$ and $\dist_{\DS}(a',b)\leq \dist_{\CS}(a',b)$, we then have that
\[ \begin{array}{ll}
\dist_{\CS}(a,a') & = \dist_{\DS}(a,a') \\
 & \leq \dist_{\DS}(a,b)+\dist_{\DS}(b,a') \\
 & \leq \dist_{\DS}(a,b)+\dist_{\CS}(b,a'). \\
 \end{array} \]
For the other sides, it remains to show without loss of generality that $\dist_{\CS}(a',b) \leq \dist_{\CS}(a',a) + \dist_{\CS}(a,b)$, or equivalently that $\dist_{\CS}(a',b) \leq \dist_{\CS}(a',a) \oplus \dist_{\CS}(a,b)$.

For this let $z=\mu(a,b) \in A\cap B$ be such that \[ \dist_{\CS}(a,b)=\dist_{\AS}(a,z)\oplus \dist_{\BS}(z,b)\]
and such that $\dist_{\AS}(a,z) + \dist_{\BS}(z,b)$ is as small as possible.

But $\dist_{\AS}(a',z) \leq \dist_{\AS}(a',a) + \dist_{\AS}(a,z)$, equivalently
\[ \dist_{\CS}(a',z) \leq \dist_{\CS}(a',a) \oplus \dist_{\CS}(a,z).\] 
Hence, \[ \begin{array}{ll}
\dist_{\CS}(a',b) & \leq \dist_{\CS}(a',z) \oplus \dist_{\CS}(z,b) \\
 & \leq (\dist_{\CS}(a',a) \oplus \dist_{\CS}(a,z)) \oplus \dist_{\CS}(z,b) \\
 & = \dist_{\CS}(a',a) \oplus ( \dist_{\CS}(a,z) \oplus \dist_{\CS}(z,b)) \\
 & =\dist_{\CS}(a',a) \oplus \dist_{\CS}(a,b).
 \end{array} \]
 This completes the proof.
\end{proof}

We are now ready to further analyze the structure of the universal spectrum without positive limits, but first some useful terminology.

\begin{defin}\label{defin:cover}
Let $\SetS$ a universal spectrum without positive limits.
\begin{enumerate}
\item For $s\in \SetS$, let $s^-$ be the largest number in $\SetS$ smaller than $s$ if $s>0$, and $0^-=0$.
\item If $s\not=\max \SetS$, then let $s^+$ be the smallest number in $\SetS$ larger than $s$, and let $s^+=s$ if $s=\max \SetS$.
\item Two numbers $s < t \in \SetS$ are said to be {\em consecutive} if $s^+=t$ (or if $t^-=s$).
\item The {\em cover} of $\{r,t\}$, is the number (in $\SetS$):
\[ \min\{s\in \SetS : \; |r-t|\leq s\}.\]
\item The {\em gap} at $s \in \SetS$, denoted by $\mathrm{gap}(s)$, is the number (in $\Reals$):
\[ \min\{|s-t| : \; t \in \SetS \setminus \{s\} \}. \]
\item If $T\subseteq \SetS_{>0}$, then $\mathrm{gap}(T)=\min\{\mathrm{gap}(t) : t\in T\}$.
\end{enumerate}
\end{defin}

Note that if $\{s,t\}\in \SetS$ with $s\not=t$, then $\mathrm{gap}(s) \leq |s-t|$. Hence, we immediately have the following general fact.

\begin{fact}\label{fact:gap}
Let $\SetS$ be a universal spectrum without positive limits, and $\MS$ be a metric space with $\spec(\MS)\subseteq \SetS$.
However, if $\dist(x,y)<\mathrm{gap}(\dist(x,z))$ for any three distinct points $\{x,y,z\}\subseteq M$, then $\dist(x,z)=\dist(y,z)$.
\end{fact}

\begin{lem}\label{lem:jump1cl}
If $\SetS$ is a universal spectrum without positive limits, then the cover $c$ of two consecutive numbers $r<t \in \SetS$ with $r+r\geq t$ is an initial number of $\SetS$.
\end{lem}

\begin{proof}
Note that $c\leq r$ because $r+r\geq t$ and $c$ is the smallest number in $\SetS$ with $r+c\geq t$.
Assume that $c$ is not initial. Then there exists a number $p<c$ with $p+p\geq c$; this implies that the two triangles, $\mathrm{T}_0$ with side lengths $p,p,c$, and $\mathrm{T}_1$ of side lengths $\{r,t,c\}$, are metric. These two triangles form an amalgamation instance via the common side $c$. Hence, because $\SetS$ satisfies the 4-values condition there exists a number $s \in \SetS$ and a metric space $\MS \in \amalg(\mathrm{T}_0,\mathrm{T}_1)$ with amalgamation distance $s$.

This is not possible. Indeed first note that $r+p<t$ because $p<c$ and again $c$ is the smallest number in $\SetS$ with $r+c\geq t$. Now, if $s\leq r$, then $s+ p \leq r+ p <t$ and the triangle $\{t,p,s\}$ is not metric. If on the other hand $s> r$, then $s \geq t$ because $r<t$ are consecutive; but now $r+ p <t \leq s$ and the triangle $\{r,p,s\}$ is not metric.
\end{proof}

This gives the following.

\begin{lem}\label{lem:lim0bl}
Let $\SetS$ be a universal spectrum without positive limits. If 0 is a limit of $\SetS$, then 0 is also a limit of the set of initial numbers of $\SetS$, and also a limit of the set of jump numbers of $\SetS$.
\end{lem}

\begin{proof}
Let $r<t<\ell \in \SetS$ be two consecutive numbers, and let $c$ be their cover.
If $r+r\geq t$, then $c\leq r$ and it follows from Lemma \ref{lem:jump1cl} that $c$ is initial. However, if $r+r < t$, then $t$ itself is initial by definition. Thus, $\SetS$ contains arbitrarily small initial numbers.

Moreover, is $s$ is an initial number, then $s^-$ is a jump number. We therefore have that $\SetS$ contains arbitrarily small jump numbers as well.
\end{proof}

\subsection{$\US$ where $\SetS$ has no positive limits} 

In this subsection we continue with $\SetS$ a universal spectrum without positive limits, but we will focus on $\US$ the homogeneous structure with spectrum $\SetS$ and construct some well chosen automorphisms.

Recall that if $s>0$ is a jump number of $\SetS$, then the relation $\stackrel{s}{\sim}$ given by $x\stackrel{s}{\sim}y$ if $\dist(x,y)\leq s$ is an equivalence relation on $\US$. If $E$ is an $\stackrel{s}{\sim}$ equivalence class and $\MS$ the metric space induced by $E$, then $\spec(\MS)=\{r\in \SetS : r\leq s\}=\SetS_{\leq s}$. On the other hand, it is evident that $\spec(\MS)$ satisfies the 4-values condition, and hence, it is a universal spectrum. It follows that $\MS$ is isomorphic to the universal homogeneous metric space $\U_{\SetS_{\leq s}}$.

We now define the notion of dense subset of $\US$, and show that similar to the rationals, if $\US$ is partitioned into finitely many dense sets, then there is a non-trivial automorphism preserving the partition.

\begin{defin}\label{defin:densesu}
Let $\SetS$ be a universal spectrum. A subset $A$ of $\US$ is {\em dense} if $A\cap E\not=\emptyset$ for every jump number $s \in \SetS_{>0}$ and for every equivalence class $E$ of the relation $\stackrel{s}{\sim}$.
\end{defin}

We will in fact build a non-trivial automorphism which is an involution.

\begin{lem}\label{lem:denssubhom}
Let $\SetS$ be a universal spectrum without positive limits but with 0 as a limit, and let $\{A_i : i \in n \}$ form a finite partition of $\US$ into dense sets. 
It then follows that there exists, for every number $s \in \SetS_{>0}$, an automorphism $f$ of $\US$ such that:
\begin{enumerate}
\item $f$ preserves the partition, that is $f[A_i]=A_i$ for every $i\in n$, and
\item $f(f(x))=x$ and $\dist(x,f(x))=s$ for all $x \in \US$.
\end{enumerate}
\end{lem}

\begin{proof}
The proof is an inductive construction on the countable domain of $\US$, and is a consequence of the following claim handling the inductive step.

\begin{claim*}
Let $\AS$ and $\BS$ be two disjoint and finite subspaces of $\US$ for which there exists an automorphism $g$ of the subspace induced by $A\cup B$ so that:
\begin{enumerate}
\item $g$ preserves the partition restricted to $A \cup B$, and
\item $g(x) \in B$, $g(g(x))=x$, and $\dist(x,g(x))=s$ for all $x \in A $.
\end{enumerate}
Let $u \in \US \setminus (A\cup B)$. Then there exists a point $v\in \US$ and an automorphism $g'$ of the subspace induced by $A\cup B\cup \{u,v\}$ so that:
\begin{enumerate}
\item $g'$ extends $g$, that is $g'(x)=g(x)$ for all $x\in A\cup B$,
\item $\dist(u,v)=s$,
\item $u$ and $v$ are in the same member of the partition, and
\item $g'(u)=v$ and $g'(v)=u$.
\end{enumerate}
\end{claim*}

To prove the claim, we first show that there exists a metric space $\MS$ with $M=A\cup B\cup \{u,v\}$ so that the following hold.
\begin{enumerate}
\item $\MS$ restricted to $A\cup B\cup \{u\}$ is equal to $\US$ restricted to $A\cup B\cup \{u\}$.
\item $\dist_{\MS}(u,v)=s$.
\item $\dist_{\MS}(v,x)=\dist(u,g(x))$ (and so $\dist_{\MS}(v,g(x))=\dist(u,x))$ for all $x\in A\cup B$.
\end{enumerate}
To verify that $\MS$ will indeed be a metric space under these conditions, it suffices that every triangle of $\MS$ is metric. Let $\{x,y,z\}$ be a triangle of $\MS$. If $v\not\in \{x,y,z\}$, then the triangle $\{x,y,z\}$ is metric because every triangle of $\US$ is metric. Now let $\{x,y,v\}$ be a triangle of $\MS$ with $u\not\in \{x,y\}$; but the triangle $\{g(x),g(y),u\}$ is metric and has the same side lengths as the triangle $\{x,y,v\}$, hence, the latter is metric. Let $\{x,u, v\}$ be a triangle of $\MS$. The sides have lengths $\dist(x,u)$, $\dist_{\MS}(x,v)=\dist(g(x),u)$, and $s$; but the triangle $\{x,g(x),u\}$ is metric and has the same side lengths.

Now the bijection $\tilde{g}$ of $A\cup B\cup \{u,v\}$ extending $g$ and interchanging $u$ and $v$ is an automorphism of $\MS$ because $\dist_{\MS}(u,x)=\dist(u,x)=\dist_{\MS}(v,g(x))$ for all $x\in A\cup B$.

Because $\US$ is homogeneous there exists an embedding $h$ of $\MS$ into $\US$ with $h(x)=x$ for all $x\in A\cup B\cup \{u\}$. By Lemma \ref{lem:lim0bl}, let $0<r\in \SetS$ be a jump number with $0<r<\mathrm{gap}(\{\dist(h(v),x) : x\in A\cup B\cup \{u\}\}$, and let $E$ be the $\stackrel{r}{\sim}$ equivalence class containing the point $h(v)$. If $i\in n$ is such that $h(v)\in A_i$, then choose $w\in A_i\cap E$; this is possible because $A_i$ is assumed to be dense. It follows from the choice of $r$ and from Fact \ref{fact:gap} that $\dist(w,x)=\dist(h(v),x)=\dist_{\MS}(v,x)$ for all $x\in A\cup B\cup \{u\}$.

The required automorphism $g'$ is simply the map corresponding to $\tilde{g}$ interchanging $u$ and $w$.
\end{proof}

\section{Distinguishing Number of Homogeneous Urysohn Metric Spaces}\label{dsec}

In this section, we show that the distinguishing number $\US$ is either 2 or infinite for any countable universal spectrum $\SetS$.
When it is 2, we will show this is so by decomposing $\US$ into a rigid subspace particularly constructed so that all automorphisms fixing this subspace also fix its complement.

% \subsection{Rigid Forest}
%RIGID FOREST
In all cases but one, the rigid subspace is made from a rigid forest. First we show how to use the graph structure of a metric space.

\begin{defin}\label{def:sdistgraph}
Consider a metric space $\MS$ with distances in $\SetS \subseteq \Reals$. For $s\in \SetS$, the {\em $s$-distance graph} of $\MS$ is the (simple) graph on the elements of $\MS$ (as vertices), and two vertices are adjacent if and only if their distance is $s$.
\end{defin}

This following observation will play a crucial role in building rigid subspaces.

\begin{obs*}
If the $s$-distance graph of a metric space is rigid, then the metric space is rigid.
\end{obs*}

\subsection{Basic Construction}

Here is the first such construction.

\begin{lem}\label{lem:unbninsin}
If  $\SetS$ be a universal spectrum, then the distinguishing number of $\US$ is 2 if there exists a positive number $s\in \SetS$ for which the following hold.
\begin{enumerate}
\item The element $s$ is not a jump number; that is, there exists a number $r\in \SetS$ with $s<r\leq s+s$.
\item For every positive $t\in \SetS$, there exist numbers $\{h_t,k_t\}\subseteq \SetS$ so that:
\begin{enumerate}
\item $s<h_t<k_t$.
\item $h_t+k_t\geq t\geq k_t-h_t$.
\end{enumerate}
\end{enumerate}
\end{lem}

\begin{proof}
We say that a set $P$ of pairs of points in $\US \setminus \MS$ is {\em stabilized by the subspace $\MS$ of $\US$} if for all $(x,y)\in P$:
\begin{enumerate}
\item There exists a point $z\in M$ with $\dist(x,z)\not=\dist(y,z)$.
\item The $s$-distance graph of $\MS$ is a rigid forest.
\end{enumerate}

The proof of the lemma is an inductive construction on the countable domain of $\US$, building such an $\MS$ as an increasing union of finite spaces such that any pair $(x,y) \in \US \setminus \MS$  is stabilized by the subspace $\MS$. We then have that the partition  of $\MS$ and its complement shows that the distinguishing number of $\US$ is 2. 

The construction of $\MS$ is a consequence of the following claim handling the inductive finite stages.

\begin{claim*} Let $P$ be a set of pairs of points in $\US$ which is stabilized by the finite subspace $\MS$ of $\US$, and let $\{x,y\}$ be two points in $\US \setminus \MS$. We then have that there exists a finite subspace $\NS$ of $\US$ containing $\MS$ and stabilizing $P'=P\cup \{(x,y)\}$.
\end{claim*}

To prove the claim, if there already exists a point $z\in M$ with $\dist(x,z)\not=\dist(y,z)$, then we can let $\NS=\MS$. Hence, we assume that $\dist(x,z)=\dist(y,z)$ for all $z\in M$.

Let $r\in \SetS$ with $s<r\leq s+s$. Let $\CS$ be a metric space with spectrum $\{s,r\}$ whose $s$-distance graph is a rigid tree $G$ which is not isomorphic to one of the trees of the $s$-distance graph of the space $\MS$. Let $e$ be an endpoint of the tree $G$. Let $\TS$ be the metric space with $T=\{x,y,e\}$ so that $\dist_\mathrm{T}(x,y)=t$ and $\dist_T(x,e)=h_t$ and $\dist_\mathrm{T}(y,e)=k_t$ as per the hypothesis. Then $\TS$ is indeed a metric space because $h_t+k_t\geq t\geq k_t-h_t$ and $h_t<k_t$. The pair of metric spaces $(\TS,\CS)$ forms an amalgamation instance. Now because $s<r$ and $s<h_t$, then $s< r'=\min\{r,h_t\}$. It follows from Lemma~\ref{lem:amalbetter} that there exists a metric space $\CS'\in \amalg_\SetS(\mathrm{T},\CS)$ with $\dist_{\CS'}(v,x)\geq r'$ and with $\dist_{\CS'}(v,y)\geq r'$ for all $v\in C$ and so that $\dist_{\CS'}(v,x)\leq \dist_{\CS'}(v,y)$ for every $v\in C$.

Now let $\MS'$ be the subspace of $\US$ induced by the set $M\cup \{x,y\}$ of points. The metric spaces $\MS'$ and $\CS'$ form an amalgamation instance. It follows from Lemma ~\ref{lem:amalbetter} again that there exists a metric space $\MS''\in \amalg_\SetS(\MS',\CS')$ so that $\dist_{\MS''}(v,z)\geq r'>s$ for every $v\in C$ and every $z\in M$.

Because $\US$ is homogeneous there exists an embedding $f$ of $\MS''$ into $\US$ with $f(z)=z$ for all $z\in M'$. Then the image $\NS$ of $\MS''$ under the embedding $f$, removing $x$ and $y$,  is as required.
\end{proof}

This yields the following case.

\begin{coroll}\label{cor:poslimno}
Let $\SetS$ be a universal spectrum. 
If $\SetS$ contains a positive number $s$ which is not a jump number, and has a limit $r$ (not necessarily in $\SetS$) with $s<r$, then the distinguishing number of $\US$ is 2.
\end{coroll}

\begin{proof}
By Lemma \ref{lem:unbninsin}, if suffices to show that for every positive $t\in \SetS$ there exist numbers $\{h_t,k_t\}\subseteq \SetS$ so that $s<h_t<k_t$, and $h_t+k_t\geq t\geq k_t-h_t$.

If $t\leq r$, then because $r$ is a limit point of $\SetS$ one can find the required numbers $s< h_t< k_t \in \SetS$ (close enough to $r$). If $t>r$, then choose $h_t \in \SetS$ close enough to $r$ so that $s<h_t<t$, and let $k_t=t$; again this is possible because $r$ is a limit point of $\SetS$.
\end{proof}

%%%%%%%%%%
% positive limit not necessarily in $\SetS$

The case of $\SetS$ having a positive limit in $\SetS$ can be handle in a similar manner, but the more general case of the positive limit not necessarily in $\SetS$ is more delicate. In that case the rigid forest will be replaced by a ``crab nest''.

 \subsection{Crab Nest}

%%%%%SPIDER and CRAB
In the more general situation of $\SetS$ having a limit point not in $\SetS$ and all points below are jump numbers, we may not be able to retain the connected components in the intended rigid $s$-graph, and therefore, we need a new structure.

First call a finite graph $S$ a {\em spider} if it is a tree which contains exactly one vertex, the {\em centre} of $S$, of degree larger than 2 and then all of the other vertices have degree 2 or are endpoints having degree one.

Two cliques $\CS$ and $\CS'$ of a graph $G$ are called {\em adjacent} if they are vertex disjoint, the order of one of them (say $\CS'$) is less than or equal to the order of the other ($\CS$), and there exists an injection $f$ of $V(\CS')$ to $V(\CS)$ so that a vertex $x\in V(\CS')$ is adjacent to a vertex $y\in V(\CS)$ if and only if $f(x)=y$.

A finite graph $G$ is a {\em crab} if there exists a rigid spider $S$ and an integer $n \geq 5$, the {\em heft} of $G$ denoted by $\mathrm{heft}(G)$, so that for the set $\mathcal{C}$ of maximal cliques of the graph $G$:
\begin{enumerate}
\item Every vertex of $G$ is a vertex of exactly one of the cliques in $\mathcal{C}$.
\item Exactly one of the cliques in $\mathcal{C}$, the {\em centre clique of $G$}, has order $n+1$ and all other cliques in $\mathcal{C}$ have order $n$.
\item If $\CS$ is the centre clique, then there exists for every vertex $x\in V(\CS)$ exactly one clique $\CS'\in \mathcal{C}$ which is adjacent to $\CS$ and for which no vertex in $V(\CS')$ is adjacent to $x$.
\item There exists a bijection $f$ of $\mathcal{C}$ to $V(S)$ which maps the centre clique of $G$ to the centre of $S$, and such that $\CS$ and $\CS'$ are adjacent if and only if $f(\CS)$ is adjacent to $f(\CS')$.
\end{enumerate}
Note that the degree of the centre of the spider $S$ is equal to $n+1$. A clique $\CS\in \mathcal{C}$ is an {\em end clique} of the crab $G$ if $f(\CS)$ is an endpoint of the spider $S$.

\begin{lem}\label{lem:crabrig}
Every crab $G$ is rigid.
\end{lem}
\begin{proof}
Let $G$ be a crab with associated spider $S$, heft $n$ and $\mathcal{C}$ the set of maximal cliques. Let $g$ be an automorphism of $G$. Then $g$ induces a permutation $\mathcal{C}$, and fixes the centre clique $\CS$ because it is the only one of size $n+1$; hence, $g$ induces an automorphism of its associated spider and because the latter is assumed to be rigid then $\mathcal{C}$ are actually fixed.

We claim that $g$ is the identity map. First let $x\in V(\CS)$, and assume that $x\not=f(x)$. There exists a (unique) clique $\CS'$ which is adjacent to the clique $\CS$ and which does not contain a vertex which is adjacent to $f(x)$. We then have that $V(\CS')$ contains a vertex $x'$ which is adjacent to $x$, implying that $f(x')\not\in V(\CS')$, a contradiction. Hence, $f(x)=x$ for all $x\in V(\CS)$, implying inductively on the distance from the centre that $f(x)=x$ for all $x\in V(G)$.
\end{proof}

The rigid subspaces we are looking for will be build of crabs into what we call crab nests.

\begin{defin}\label{defin:crabnest}
A graph $G$ is a {\em crab nest} if it has a subgraph $H$ on the same vertices as $G$, the {\em crab graph} of $G$, together with an enumeration of its connected components $\{H_i, i\in I\}$ for $I=\omega $ or $I=n\in \omega$, and if there exists a set $R\subseteq V(G)$, the {\em distinguished endpoint set} of $G$, so that for every $i\in I$:
\begin{enumerate}
\item $H_i$ is a crab, and an induced subgraph of $G$.
\item $\mathrm{heft}(H_i)+2<\mathrm{heft}(H_{i+1})$.
\item $R$ contains exactly one vertex $r_i \in V(H_i)$, and this vertex is a vertex of an end clique of $H_i$.
\item If $(x,y)\in E(G)\setminus E(H)$, then the following hold.
\begin{enumerate}
%\item $\{x,y\}\not\subseteq V(H_i)$.
\item If $x\in V(H_i)$ and $y\in V(H_j)$ for some $j<i$, then $x=r_i$. \label{Item:4b}
\item If $x=r_i$, $(r_i,z)\in E(G)\setminus E(H)$ and $z\in H_k$ with $k<i$, then $y=z$. \label{Item:4c}
\end{enumerate}
\end{enumerate}
\end{defin}

\begin{lem}\label{lem:spidwrig}
Every crab nest is rigid.
\end{lem}
\begin{proof}
Let $G$ be a crab nest with crab graph $H$, the corresponding enumeration $\{H_i, i\in I\}$ of the connected components of $H$, corresponding endpoint set $R$. Let $f$ be an automorphism of $G$ and we show that $f$ must be the identity.

Assume that there exists a vertex $x\in V(H_i)$ with $i>0$ such that $f(x)\in V(H_0)$. Let $\CS$ be the maximal clique (thus, of size at least three) of $H_i$ with $x\in V(\CS)$. Let
\[
j=\max\{k\in I : \text{ for some } y\in V(\CS) ,  f(y)\in V(H_j)  \}.
\]
Assume that $j>0$. Note that if $y\in V(\CS) $ is such that $ f(y)\in V(H_j)$, then $(f(y),f(x) )\in E(G)\setminus E(H)$, and thus, by item (\ref{Item:4b}) there can be only one such element $y$ and we must have $f(y)=r_j$. We then have that $f(z)\in H_k$ for some $k<j$ for all other $z\in V(\CS)$, now contradicting to item (\ref{Item:4c}) because we would have both $(f(y),f(x))$ and $(f(z),f(x))$ in $E(G)\setminus E(H)$.

It follows that $j=0$. That is, the automorphism $f$ maps every element of $V(\CS)$ into $V(H_0)$. But every triangle of $H_0$ is in one of the cliques of $H_0$. Implying that $f$ maps the clique $\CS$ into a clique of $H_0$. But this is not possible because $\mathrm{heft}(H_i)>\mathrm{heft}(H_0)+2$. It follows that $f$ does not map any vertex of $V(G)\setminus V(H_0)$ into $V(H_0)$. Implying that if $x\in V(H_0)$, then $f(x)\in V(H_0)$.

This in turn implies, because the crab $H_0$ is rigid, that $f(x)=x$ for all $x\in V(H_0)$. Then via induction on the index set $I$ it follows that $f(x)=x$ for all $x\in V(G)$.
\end{proof}

 We are now ready to handle the case of $\SetS$ having a positive limit not necessarily in $\SetS$.

\begin{lem}\label{lem:poslimno2}
Let $\SetS$ be a universal spectrum. If $\SetS$ has a positive limit (not necessarily in $\SetS$), then the distinguishing number of $\US$ is 2.
\end{lem}

\begin{proof}
If $\SetS$ has a limit $r$ (not necessarily in $\SetS$) and one can find a non-jump number $s<r$, then Corollary
\ref{cor:poslimno} applies. This is the case if $\SetS$ has two positive limits $r'<r$, in which case one can find such an $s$ close to $r'$. Similarly, this is the case if  $r$ is a limit of the elements of $\SetS$ below $r$.

Hence, we may assume that $\SetS$ has only one positive limit $r$, every number in $\SetS$ less than $r$ is insular, and the elements of $\SetS$ above $r$ form a (possibly two way) sequence converging to $r$. In particular this means that every non-empty and bounded above subset of $\SetS$ has a maximum, that is $\SetS$  is inversely well ordered, and thus, the operation $\oplus$ is defined for $\SetS$.

\medskip
We are now ready to undertake the construction of the rigid subspace using a rigid $s$-graph, where $s$ will be chosen close enough to $r$ and such that $r<s<r+r$. That space and its complement will show that the distinguishing number of $\US$ is 2.

For the sake of this proof, we say that a set $P$ of pairs of points in $\US \setminus \MS$ is {\em stabilized} by the subspace $\MS$ if the following properties hold.
\begin{enumerate}
\item The $s$-distance graph of $\MS$ is a crab nest $G$ with crab graph $H$ and enumeration $H_0,H_1,H_2,\dots$ of the connected components of $H$ and distinguished endpoint set $R$ of $G$.
\item If $u\not=v$ are two points of $\MS$, then $\dist(u,v) \geq r$.
\item For every pair $(x,y)\in P$, there exists a crab $H_i$ and point $r_i\in R\cap H_i$ such that
$\dist(x,r_i)\not=\dist(y,r_i)$. % and $\max\{ \dist(x,r_i), \dist(y,r_i)\} <r+r$.
\end{enumerate}

\noi The proof of Lemma \ref{lem:poslimno2} is an inductive construction on the countable domain of $\US$, and is a consequence of the following claim handling the inductive step.

\begin{claim*} Let $P$ be a set of pairs of points in $\US$ which is stabilized by the finite subspace $\MS$ of $\US$. Let $\{x,y\}$ be two points in $\US \setminus \MS$ such that $\dist(x,y)=t>0$, and $\dist(x,z)=\dist(y,z)$ for all $z\in M$. We then have that  there exists a finite subspace $\NS$ of $\US$ containing $\MS$ and stabilizing $P'=P\cup \{(x,y)\}$.
\end{claim*}

To prove the claim, let $G$ be the $s$-distance graph of $\MS$ with crab graph $H$ and enumeration $\{H_i: i\in n\}$ of the connected components of $H$ and distinguished endpoint set $R$.

If $t<r$, then let $r<h_t<k_t$ be two numbers in $\SetS$ (close to $r$) with $k_t-h_t\leq t$. Otherwise, consider $s'<s'' \in \SetS$ such that $r<s<s'<s''<r+r$ (this is why we picked $s$ close enough to $r$). If $r\leq t<r+r$, then let $h_t=s'$ and $k_t=s''$; if $r+r \leq t$, then let $h_t=s'$ and $k_t=t$.

Let $\CS$ be a metric space with spectrum $\{r,s\}$ whose $s$-distance graph is a crab $H_n$ for which $\mathrm{heft}(H_{n-1})+2<\mathrm{heft}(H_n)$. Let $r_n$ be a vertex of an end clique of the crab $H_n$ and let $R'=R\cup \{r_n\}$. Let $\TS$ be the metric space with $T=\{x,y,r_n\}$ so that $\dist_\mathrm{T}(x,y)=t$ and $\dist_T(x,r_n)=h_t$ and $\dist_\mathrm{T}(y,r_n)=k_t$. Note that $\TS$ is in all cases indeed a metric space. The pair of metric spaces $(\TS,\CS)$ forms an amalgamation instance. It follows from Lemma~\ref{lem:amalbetter} and from $s' \leq r+r \leq r +h_t <r +k_t$ that there exists a metric space $\CS'\in \amalg_\SetS(\TS,\CS)$ with
\begin{align}\label{E:1}
\text{$\dist_{\CS'}(v,x)\geq s'$ and $\dist_{\CS'}(v,y)\geq s'$ for all $r_n\not=v\in C$} \tag{i.}
\end{align}
and so that $\dist_{\CS'}(v,x)\leq \dist_{\CS'}(v,y)$ for every $v\in C$.

%If $t\geq r$ then $r\leq h_t<k_t$ and it follows from Lemma~\ref{lem:amalbetter} that:
%\begin{align} \label{E:2}
%s<r\leq \dist_{\CS'}(e,x)<\dist_{\CS'}(e,y). \tag{ii.}
%\end{align}

 Now let $\MS'$ be the subspace of $\US$ induced by the set $M\cup \{x,y\}$. The metric spaces $\MS'$ and $\CS'$ form an amalgamation instance, and let $\MS''= \amalg^\oplus_\SetS(\MS',\CS')$ provided by Lemma \ref{lem:oplusamalg}. We then have that for $v\in C$ and $z \in M$:
 \begin{align}\label{E:3}
 \dist_{\MS''}(v,z) \geq \dist_{\CS'}(v,x)\oplus \dist(x,z)\geq \dist_{\CS'}(v,x). \tag{iii.}
 \end{align}
 Hence, for $v\in C$ and $z\in M$, if $t\geq r$ or if $v\not=r_n$, then
 \begin{align} \label{E:4}
\dist_{\MS''}(v,z)\geq s' >s, \tag{iv.}
 \end{align}
and if $t<r$, then
 \begin{align}\label{E:5}
 \dist_{\MS''}(r_n,z)=h_t\oplus \dist(x,z). \tag{v.}
 \end{align}

We claim that the $s$-distance graph of $\MS'' \setminus \{x,y\}$ is a crab nest $G''$ with distinguishing endpoint set $R'$. The crab graph $H''$ of $G''$ is the crab graph $H$ of $G$ together with the additional connected component $H_n$. Item (4) of Definition \ref{defin:crabnest} remains to be verified. Let $(z,v) \in E(G'')\setminus E(H'')$.  It follows  that  $z$ and $v$ are two points of $\MS''$ with $\dist_{\MS''}(z,v)=s$. If $\{z,v\}\subseteq M$, then Item (4) of Definition \ref{defin:crabnest} will be satisfied. If $\{z,v\}\subseteq C$, then $\{z,v\}\in E(H_n)$ because the spectrum of $\CS$ is $\{s,r\}$ and the $s$-distance graph of $\CS$ is the crab $H_n$. Hence, we may assume that $v\in C$ and $z\in M$. If $t\geq r$ or if $v\not=r_n$ it follows from \ref{E:4} above that $\dist_{\MS''}(v,z)\geq s'>s$.

Thus, $t<r$ and $v=r_n$. If $r\leq \dist(x,z)=\dist(y,z)$, then $\dist_{\MS''}(r_n,z)=h_t\oplus \dist(x,z)\geq r\oplus r \geq s'>s$. If $p=\dist(x,z)<r$, then it is possible that $\dist_{\MS''}(r_n,z)=h_t\oplus \dist(x,z)=h_t\oplus p=s$. But now assume that $q=\dist(x,w)<r$ for a point $w\in M$. We then have that  $p$ and $q$ are insular numbers of $\SetS$ and both smaller than $r$. If $z \not=w$, then $\dist(z,w)\geq r$ by assumption, which is a contradiction because $p\oplus q=\max\{p,q\}<r$ because they are insular points. Hence, $z=w$ verifying Item (4c) of Definition \ref{defin:crabnest} and $G''$ is a crab nest.

The space $\MS''$ is not a subspace of $\US$ but otherwise meets the conditions for stabilizing the set $P\cup \{(x,y)\}$. Since $\US$ is homogeneous, there exists an embedding $f$ of $\MS''$ into $\US$ with $f(a)=a$ for all $a\in M'$. Then the image $\NS$ of $\MS''$ under the embedding $f$ is as required to prove the claim.
\end{proof}

%%%%end of crab

%%%%%%%%%%
% Remaining Cases

\subsection{The Remaining Cases}

There are a few remaining cases to handle, made possible from previous results and techniques.

\begin{lem}\label{lem:boundbel}
Let $\SetS$ be a universal spectrum. If there exist numbers $a <b \in \SetS$ with $b-a\leq \inf\SetS_{>0}$, then the distinguishing number of $\SetS$ is 2.
\end{lem}

\begin{proof}
Let $p=\inf \SetS_{>0} \geq b-a,$ where $a<b \in \SetS$. If $\SetS$ has a positive limit, then the distinguishing number of $\US$ is two according to Lemma \ref{lem:poslimno2}. Thus, we assume that $\SetS$ does not have a positive limit. In particular, $p\in \SetS$.

Moreover, the set $\SetS$ must contain an initial non-jump number. For otherwise, because $a < b\leq p+a\leq a+a=2a$ and thus, $a$ is non-jump number and must be a non-initial numbers. We then have that there is $a_1 \in \SetS \cap [a/2,a]$, a non-jump number which for the same reason must be a non-initial number. Continuing in this manner yields a positive limit in $\SetS$, a contradiction.

Now, let $s$ be the smallest positive initial non-jump number of $\SetS$, and let $r$ be the smallest number in $\SetS$ larger than $s$. We cannot have $a<s$ because being a non-jump number and small that $s$ would make $a$ again a non-initial number; and similar to above would yield $\SetS$ with a positive limit point. If $s<a$, then with $h_t=a$ and $k_t=b$ if $t\leq r$, and with $h_t=r$ and $k_t=t$ if $t>r$, the conditions of Lemma \ref{lem:unbninsin} are satisfied and hence, $\US$ has then distinguishing number 2.

Hence, we may assume that $s=a$ and $r$ can be taken for $b$, implying that $r-s\leq p$, and of course $q\geq p$ for all $q \in \SetS$. We then proceed analogously to the proof of Lemma \ref{lem:unbninsin},  and construct a subspace together with  its complement showing that the distinguishing number of $\US$ is 2.

\medskip

\noindent For the sake of this proof, we say that a set of pairs $P$ in $\US \setminus \MS$ is {\em stabilized} by the subspace $\MS$ of $\US$ if for all $(x,y)\in P$:

\begin{enumerate}
\item There exists a point $z\in M$ with $\dist(x,z)\not=\dist(y,z)$.
\item The $s$-distance graph of $\MS$ is a rigid forest.
\end{enumerate}

The proof of the Lemma is an inductive construction on the countable domain of $\US$, and is a consequence of the following claim handling the inductive step.

\begin{claim*}
Let $P$ be a set of pairs of points in $\US \setminus \MS$ which is stabilized by the finite subspace $\MS$ of $\US$. Let $\{x,y\}$ be two points in $\US \setminus \MS$ with $t=\dist(x,y)>0$. We then have that  there exists a finite subspace $\NS$ of $\US$ containing $\MS$ and stabilizing $P'=P\cup \{(x,y)\}$.
\end{claim*}

To prove the claim, if there exists a point $z\in M$ with $\dist(x,z)\not=\dist(y,z)$ let $\NS=\MS$. Hence, we may assume that $\dist(x,z)=\dist(y,z)$ for all $z\in M$.

Let $\CS$ be a metric space with spectrum $\{s,r\}$ whose $s$-distance graph is a rigid tree $S$ which is not isomorphic to one of the trees of the $s$-distance graph of the space $\MS$. Let $e$ be an endpoint of the tree $S$. If $t\leq r$ let $h_t=s$ and $k_t=r$; and if $t>r$ let $h_t=r$ and $k_t=t$. Let $\TS$ be the metric space with $T=\{x,y,e\}$ so that $\dist_{\TS}(x,y)=t$, $\dist_{\TS}(x,e)=h_t$ and $\dist_{\TS}(y,e)=k_t$. Note that $\TS$ is indeed a metric space in all cases. The pair of metric spaces $(\TS,\CS)$ forms an amalgamation instance, and it follows from Lemma~\ref{lem:amalbetter} and from $r\leq s+h_t<s+k_t$ that there exists a metric space $\CS'\in \amalg_\SetS(\TS,\CS)$ with $\dist_{\CS'}(v,x)\geq r$ and with $\dist_{\CS'}(v,y)\geq r$ for all $e\not=v\in C$ and so that $\dist_{\CS'}(v,x)\leq \dist_{\CS'}(v,y)$ for every $v\in C$. Hence, $\dist_{\CS'}(v,y)\geq \dist_{\CS'}(v,x)\geq s$ for all $v\in C$

 Let $\MS'$ be the subspace of $\US$ induced by the set $M\cup \{x,y\}$ of points. The metric spaces $\MS'$ and $\CS'$ form an amalgamation instance.
 Let $\MS''= \amalg^\oplus_\SetS(\MS',C')$. It follows from the definition of $\oplus$ that:
 \[
 \dist_{\MS''}(v,z)\geq\dist_{\CS'}(v,x)\oplus \dist(x,z)\geq s\oplus \dist(x,z)\geq r.
 \]
 Because $\US$ is homogeneous there exists an embedding $f$ of $\MS''$ into $\US$ with $f(a)=a$ for all $a\in M'$, and the image $\NS$ of $\MS''$ under the embedding $f$ is as required for the claim. \end{proof}

We arrive at the first case where $\US$ has infinite distinguishing number.

\begin{lem}\label{lem:nolim}
Let $\SetS$ be a universal spectrum which does not have a limit, then the distinguishing number of $\US$ is 2 or infinite. If $p=\min(\SetS_{>0})$, then the distinguishing number of $\US$ is 2 if and only if there exist numbers $a< b \in \SetS$ with $b-a \leq p$.
\end{lem}

\begin{proof}
If there exist number $a<b \in \SetS$ with $b-a \leq p$, then the distinguishing number of $\US$ is 2 again according to Lemma~\ref{lem:boundbel}. This is the case if $p$ is not a jump number, because if $p<q\leq p+p$, then $q-p\leq p $. Thus, we may assume that $p$ is a jump number, and hence, an insular number because being the smallest positive number of $\SetS$ it is an initial number.

Thus, we assume that $p$ is insular and $p<b-a$ for all $a<b \in \SetS$, and we show that the distinguishing number of $\US$ is not finite.
Let $\gamma$ be a colouring function of $\US$ with $n\in \omega$ colours, that is $\gamma[\US]\subseteq n$. Note that because $p$ is a jump number and $p=\min(\SetS_{>0})$, the relation $\sim$ on $\US$ given by $x\sim y$ if and only if $\dist(x,y)=p$ is an equivalence relation. Let $E$ be a $\sim$ equivalence class of $\US$. By our assumption, we have that n $\dist(x,z)=\dist(y,z)$ for all points $z$ of $\US \setminus E$ and all points $x,y\in E$. The set $E$ is infinite because every finite metric space $\MS$ with $\spec(\MS)=\{p\}$ is an element of the age $\mathfrak{A}(S)$. Hence, there are two points $x\not=y$ in $E$ with $\gamma(x)=\gamma(y)$. The function $f$ with $f(z)=z$ for all points $z$ of $\US \setminus \{x,y\}$, and $f(x)=y$ and $f(y)=x$, is then a colour preserving automorphism of $\US$.
\end{proof}

The following is another instance where we can show that the distinguishing number is infinite.

\begin{lem}\label{lem:inflim0}
Let $\SetS$ be a universal spectrum with no positive limits but with 0 as a limit.   If $\mathrm{gap}(\SetS_{\geq s})>0$ for every positive number $s \in \SetS$, then there exists for every $n\in \omega$ and every partition $\boldsymbol{P}=\{A_i: i\in n \}$ of $\US$ a non-trivial automorphism $f$ of $\US$ preserving $\boldsymbol{P}$.
\end{lem}

\begin{proof}
There exists a subset $I\subseteq n$ and a jump number $s$ of $\SetS$ and an $\stackrel{s}{\sim}$ equivalence class $E$ so that for all $i\in I$ the set $A_i$ is dense for the homogeneous metric space $\mathrm{E}$, and so that $A_i\cap E=\emptyset$ for all $i\not\in I$.

Let $r \in \SetS$ with $0<r<\mathrm{gap}(\SetS_{\geq s})$. According to Lemma \ref{lem:denssubhom}, there exists an automorphism $f'$ of $\mathrm{E}$ with $\dist(x,f'(x))=r$ for all $x\in E$ which preserves the partition of $E$ induced by the partition $\boldsymbol{P}$ of $\US$. It follows from Fact \ref{fact:gap} that $\dist(y,x)=\dist(y,f'(x))$ for all points $y$ in $\US \setminus E$ and all points $x\in E$. It follows that the function $f: \US \to \US$ with $f(x)=f'(x)$ if $x\in E$ and with $f(y)=y$ if $y\not\in E$ is a non-trivial automorphism of $\US$ preserving $\boldsymbol{P}$.
\end{proof}

We can then characterize the case of a universal spectrum with no positive limits but with 0 as a limit.

\begin{lem}\label{lem:infd0}
Let $\SetS$ be a universal spectrum with no positive limits but with 0 as a limit, then the distinguishing number of $\US$ is 2 or infinite. The distinguishing number of $\US$ is infinite if and only if $\mathrm{gap}(\SetS_{\geq s})>0$ for every positive number $s\in \SetS$.
\end{lem}

\begin{proof}
On account of Lemma \ref{lem:inflim0} it remains to prove that if there exists a positive number $s\in \SetS$ for which $\mathrm{gap}(\SetS_{\geq s})=0$, then the distinguishing number of $\US$ is equal to 2.

But if every number in $\mathrm{gap}(\SetS_{\geq s})$ is insular, then $\mathrm{gap}(\SetS_{\geq s})\geq s>0$. Let $r$ be the smallest non-insular initial number larger than or equal to $s$. Then $\mathrm{gap}(\SetS_{\geq r})=0$ and in turn then $\mathrm{gap}(\SetS_{\geq r^+})=0$. By Lemma \ref{lem:unbninsin}, the distinguishing number of $\US$ is 2.
\end{proof}

This completes all the results required to prove the Main Theorem characterizing the distinguishing number of universal homogeneous Urysohn metric spaces.

\section{Conclusion}

We have shown that the distinguishing number of every universal homogeneous Urysohn metric spaces is either 2 or infinite, and moreover characterized when each case occurs by structural properties of the corresponding universal spectrum. It is interesting that this is the case even though the permutation group of these Urysohn metric spaces is often imprimitive. This is for example the case when the spectrum  $\SetS$ contains a jump number  $s>0$, then the relation $\stackrel{s}{\sim}$ given by $x\stackrel{s}{\sim}y$ if $\dist(x,y)\leq s$ is an equivalence relation on $\US$. Hence if this is a non-trivial relation, for example if $\SetS$ contains an element larger than $s$, then the automorphism group of $\US$ is imprimitive. We thus propose an open problem suggested by the referee. 

\begin{problem*}
Characterize universal spectrum $\SetS$ such that the corresponding universal homogeneous Urysohn metric space $\US$ has a primitive automorphism group.
\end{problem*}

In these cases of imprimitivity, it is due to the homogeneity and universality that the distinguishing number passes directly from 2 to infinity. However, one cannot expect the distinguishing number of every metric space to always be either 2 or infinite, even for homogeneous metric spaces. This is the case of the pentagon $C_5$ equipped with the graph distance, making it into an homogeneous metric space with primitive automorphism group $D_{10}$ and distinguishing number 3. One can also produce an infinite homogeneous metric space of distinguishing number 3, but with imprimitive automorphism group. Indeed consider the Rado (homogeneous) graph, first turn it into a metric space $\RS$ with spectrum $\{0,3,5\}$ by assigning distance 5 to every edge, distance 3 to non-edges, and then consider the wreath product $\RS[\MS]$ for $\MS$ the metric space consisting of two points of distance 1. This creates an homogeneous metric space $\RS[\MS]$ with spectrum $\{0,1,3,5\}$. Since the Rado graph has distinguishing number 2, we must use 3 colours to obtain two different sets of two different colours to assign to elements of $\MS$. Finally, the automorphism group of $\RS[\MS]$ is imprimitive because points of distance 1 form a non-trivial equivalence relation.

\section*{Acknowledgments}
 The contributions of two excellent referees are very much appreciated and gratefully acknowledged.

%\vfill\eject

 \end{document}